\documentclass{amsart}
\usepackage{amsfonts}
\usepackage{amsmath,amssymb}
\usepackage{amsthm}
\usepackage{amscd}
\usepackage{graphics}
\usepackage{graphicx}

\theoremstyle{remark}{

\newtheorem{Prob}{{\rm Problem}}

}
\theoremstyle{plain}
{

\newtheorem{Thm}{Theorem}

}

\begin{document}
\title[Moment-like maps and algebraic functions with prescribed preimages]{Moment-like maps and real algebraic functions with prescribed preimages}
\author{Naoki kitazawa}

\address{
}
\email{naokikitazawa.formath@gmail.com}
\urladdr{https://naokikitazawa.github.io/NaokiKitazawa.html}
\maketitle
\begin{abstract}
We discuss a problem on singularity theory of differentiable (smooth) or real algebraic maps which is different from knowing existence and has been difficult: constructing explcit real algebraic functions.

We discuss construction of real algebraic functions with exactly one singular value, the singular points being of definite type, and prescribed preimages of single points.
We discuss generalizations of the canonical projection of the unit sphere around the pole and the Morse-Bott functions around the boundaries of the images with preimages diffeomorphic to the torus.
 
This has been discussed in the differentiable (smooth) category since the pioneering study of Sharko in 2006, followed by Masumoto-Saeki, Michalak and so on: the author has first considered the cases respecting the topologies of the preimages of the points where these studies had not done this essentially. Related real algebraic studies have been started by the author essentially in 2020's and the studies are developing, mainly due to the author.

\end{abstract}
\section{Introduction with short exposition on our result, fundamental terminologies, notions and notation, and the content of the present paper.}
\label{sec:1}
\subsection{Our result and some related history.}
We discuss the following problem. Some terminologies will be explained later.
\begin{Prob}
\label{prob:1}
Let $F$ be a smooth, closed and connected manifold of dimension $m-1$ with $m \geq 2$. Let $a<b$ be real numbers.
Do there exist a smooth, compact and connected manifold $\tilde{M}$ and a smooth function ${\tilde{f}}_{\tilde{M},a,b}:\tilde{M} \rightarrow \mathbb{R}$ with the following?
\begin{enumerate}
\item The function ${\tilde{f}}_{\tilde{M},a,b}:\tilde{M} \rightarrow \mathbb{R}$ is the restriction of the canonical projection ${\pi}_{m+k,1}$.
\item $\tilde{M}$ is a subset of the zero set $M$ of some real polynomial map which is also non-singular.
\item The value $b$ is the unique singular value of the function ${\tilde{f}}_{\tilde{M},a,b}$. Its image is $[a,b]=\{a \leq t \leq b\}$.
\item The  preimage ${{\tilde{f}}_{\tilde{M},a,b}}^{-1}(a)$ is diffeomorphic to $F$. 
\end{enumerate}
\end{Prob}

This has been considered in the differentiable situations since \cite{sharko}, essentially. In short, only the third and fourth conditions are respected.
\cite{masumotosaeki, michalak} follow this. They are essentially on the case $F$ is the circle $S^1$.

\cite{kitazawa1, kitazawa4} follow them by respecting the types of the preimages of the points first. For example, in the case $F$ is the unit sphere, a closed, connected and orientable surface, or a closed, connected and non-orientable surface whose genus is even, this is affirmatively shown. 

Our new result consists of several real algebraic affirmative answers. Related real algebraic theory has been launched by the author essentially and for our pioneering study, see \cite{kitazawa2} for example and see also \cite{kitazawa3}.

We can also consider the case where the unique singular value of the function is in the interior of the image, diffeomorphic to a closed interval, and not an extremum. Related cases are presented in the studies exhibited before. For example, 
\cite{kitazawa1, kitazawa4} have considered the case of Morse functions with the preimages containing no singular points of the functions being diffeomorphic to disjoint unions of spheres $S^{m-1}$ and products of two spheres $S^{k} \times S^{m-k-1}$ ($1 \leq k<m-1$), in the differentiable (smooth) category. Some real algebraic result is also in our preprint \cite{kitazawa7} and see also \cite{kitazawa3} again. We do not consider such cases in our paper.

In our general viewpoint, our study is also on a branch of singularity theory of differentiable (smooth) maps and real algebraic ones and applications to differential topology and real algebraic geometry. We can expect related contributions in the future. Especially, existence of such maps and manifolds is important of course and construction of these objects is another problem and has been also important and difficult. We are deeply concerned with such a problem on construction.
\subsection{Fundamental terminologies, notions and notation.}
We review and introduce several terminologies, notions and notation.
We use ${\mathbb{R}}^k$ for the $k$-dimensional Euclidean space where $\mathbb{R}:={\mathbb{R}}^1$. Let $||x|| \geq 0$ denote the distance of $x \in {\mathbb{R}}^k$ and the origin $0 \in {\mathbb{R}}^k$ under the standard Euclidean metric. 
Let ${\pi}_{k,k_1}:{\mathbb{R}}^{k} \rightarrow {\mathbb{R}}^{k_1}$ denote the map mapping each point $x=(x_1,x_2) \in {\mathbb{R}}^{k_1} \times {\mathbb{R}}^{k_2}={\mathbb{R}}^k$ to $x_1 \in {\mathbb{R}}^{k_1}$ with the conditions $k_1, k_2>0$ and $k=k_1+k_2$: this is the canonical projection. The canonical projection of the ($k-1$)-dimensional unit sphere $S^{k-1}:=\{x \in {\mathbb{R}}^k \mid ||x||=1\}$ is the restriction ${\pi}_{k,k_1} {\mid}_{S^{k-1}}$. Let $D^k:=\{x \in {\mathbb{R}}^k \mid ||x|| \leq 1\}$ denote the $k$-dimensional unit disk.

For a smooth manifold $X$, let $T_xX$ denote the tangent vector space at $x$. For a smooth map $c:X \rightarrow Y$, the differential ${dc}_x:T_xX \rightarrow T_{c(x)}Y$ at $x$ is a linear map and if the rank drops, then $x$ is a {\it singular} point of $c$ and $c(x)$ is a {\it singular} value of $c$.

Our {\it real algebraic} manifold is a union of connected components of the zero set of a real polynomial map. A {\it non-singular} real algebraic manifold is defined bases on the implicit function theorem for the real polynomial map. The real affine space and the unit sphere are of simplest examples of non-singular real algebraic manifolds. Our {\it real algebraic} map is the composition of the canonical embedding into some real affine space with the canonical projection.
Let $\mathbb{N} \subset \mathbb{R}$ denote the set of all positive integers. For a real number $a$, let ${\mathbb{N}}_a:=\{i \in \mathbb{N} \mid i \leq a \}$ (e.g. for $i \in \mathbb{N}$, ${\mathbb{N}}_i:=\{1,\cdots,i\}$ ).
\subsection{The content of the present paper.}
In the next section, we define our {\it moment-like} maps, reviewing and respecting our preprint \cite{kitazawa7}.
The third section shows our main result explicitly. Our result is some positive answers to Problem \ref{prob:1} and related corollary.
The fourth section is devoted to the proof of our main result. The fifth section is devoted to another additional result, answering Problem \ref{prob:1} again. The fifth section is also devoted to future problems.
\section{Moment-like maps.}
A {\it Morse} function $c_0$ is a smooth function such that both the following hold.
\begin{itemize}
\item Each singular point of the function is in the interior of the manifold of the domain.
\item At each singular point of it, for suitable local coordinates respecting the order of the real numbers (or $\mathbb{R}$), it is represented by the form $c_0(x_1,\cdots,x_m)={\Sigma}_{j=1}^{m-i(p)} {x_j}^2-{\Sigma}_{j=1}^{i(p)} {x_{m-i(p)+j}}^2+c_0(p)$.
\end{itemize}
A {\it Morse-Bott} function $c$ is a smooth function such that both the following hold.
\begin{itemize}
\item Each singular point of the function is in the interior of the manifold of the domain.
\item For suitable local coordinates, it is represented as the composition of a smooth map with no singular points with a {\it Morse} function $c_0$.
\end{itemize}

For each singular point of the Morse function $c_0$, $i(p)$ can be defined as the unique integer $0 \leq i(p) \leq m$ and called the {\it index} of $p$ {\it for $c_0$}.

A singular point of a Morse function is also characterized in the following way. A singular point of a smooth function is represented as a singular point of some Morse function if and only if the value of the Hessian $\det {(\frac{\partial c_0}{\partial x_i \partial x_j})}_{(i,j) \in {\mathbb{N}}_m \times {\mathbb{N}}_m}$ there is not $0$ for some local coordinates.

The canonical projection ${\pi}_{m,1} {\mid}_{S^{m-1}}$ of the unit sphere is of simplest Morse functions.

Morse-functions are regarded as Morse-Bott functions by our definitions. For Morse functions, see \cite{milnor} and for Morse-Bott functions, see \cite{bott}, for example.

Hereafter, we implicitly apply fundamental arguments in real algebraic geometry, explained in \cite{bochnakcosteroy, kollar} and singularity theory of differentiable maps, explained in \cite{golubitskyguillemin}, for example.

See also our preprints \cite{kitazawa7, kitazawa8} where we do not assume the arguments and results of these preprints of the author. For example, Theorem \ref{thm:1} is our main result of \cite{kitazawa7} and proven here in a shorter argument where essential arguments are same as those in the original study.

\begin{Thm}
\label{thm:1}
Let $\{f_j\}_{j=1}^{l_1}$ be a family of real polynomial functions such that $S_j$ is a disjoint union of some connected components of the zero set of $f_j$ and that $S_j$ is non-singular. Let $m_{l_1,l_2}:{\mathbb{N}}_{l_1} \rightarrow {\mathbb{N}}_{l_2}$ be a surjective map and $m_{l_2}$ an integer-valued map on ${\mathbb{N}}_{l_2}$ whose values are always non-negative. Let $D \subset {\mathbb{R}}^n$ be a non-empty open set where we abuse $\overline{D}$ for its closure in ${\mathbb{R}}^n$. Assume also the following.
\begin{enumerate}
\item \label{thm:1.1} The relation $D={\bigcap}_{j=1}^{l_1} \{x \in {\mathbb{R}}^n \mid f_j(x)>0 \}$ and $S_j \bigcap \overline{D}=\overline{D} \bigcap \{x \in {\mathbb{R}}^n \mid f_j(x)=0\}$ hold.
\item \label{thm:1.2} The set $S_j \bigcap \overline{D}$ is non-empty for each $1 \leq j \leq l_1$.
\item \label{thm:1.3} For any integer ${l_1}^{\prime}$ satisfying $1 \leq {l_1}^{\prime} \leq l_1$, and any increasing sequence $\{k_j\}_{j=1}^{{l_1}^{\prime}}$ of integers satisfying $1 \leq k_j \leq l_1$ with ${\bigcap}_{j=1}^{{l_1}^{\prime}} S_{k_j} \bigcap \overline{D}$ being non-empty, a condition {\rm (}\ref{thm:1.3.1}{\rm )} on the transversality holds, and an additional condition {\rm (}\ref{thm:1.3.2}{\rm )} on the map $m_{l_1,l_2}$ holds. 
\begin{enumerate}
\item \label{thm:1.3.1} For any point $p \in {\bigcap}_{j=1}^{{l_1}^{\prime}} S_{k_j} \bigcap \overline{D}$, the dimension of ${\bigcap}_{j=1}^{{l_1}^{\prime}} T_p  S_{k_j}$ is $n-{l_1}^{\prime}$.
\item \label{thm:1.3.2} The restriction of $m_{l_1,l_2}$ to the set $\{k_j\}_{j=1}^{{l_1}^{\prime}}$ is injective.
\end{enumerate}
\end{enumerate}
Then the set

$$M:=\{(x,y)=(x,\{y_{i_1,i_2}\}_{(i_1,i_2) \in \{(a,b) \in {\mathbb{N}}_{l_2} \times \mathbb{N} \mid 1 \leq b \leq m_{l_2}(a)+1\}}) \in {\mathbb{R}}^n \times {\prod}_{i_1=1}^{l_2} {\mathbb{R}}^{m_{l_2}(i_1)+1} \mid f_i(x)-{||y_{i}||}^2=0, i \in {\mathbb{N}}_{l_2}\}$$
 is the zero set of the real polynomial map being also non-singular {\rm (}we use the notation $y_i=(y_{i,1} \cdots y_{i,m_{l_2}(i)+1})${\rm )}. This is of {\rm (}${\Sigma}_{j=1}^{l_2} (m_{l_2}(j))+n${\rm )}-dimensional. 
We call the map ${\pi}_{{\Sigma}_{i=1}^{l_2} (m_{l_2}(i)+1),n} {\mid}_M$ the {\rm moment-like map reconstructed from} $(D,\{f_j\}_{j=1}^{l_1},\{S_j\}_{j=1}^{l_1},m_{l_1,l_2},m_{l_2})$ and let it be denoted  by $f_{(D,\{f_j\}_{j=1}^{l_1},\{S_j\}_{j=1}^{l_1},m_{l_1,l_2},m_{l_2})}$. It also follows that $f(M)=\overline{D}$.
\end{Thm}
\begin{proof}[{\rm (}A sketch of{\rm )} another exposition of our proof.]
Let $1 \leq {l_1}^{\prime} \leq l_1$ be an integer. Consider any point $p \in {\bigcap}_{j=1}^{{l_1}^{\prime}} S_{k_j} \bigcap \overline{D}$ which is not in any space represented in the form ${\bigcap}_{j=1}^{{l_1}^{\prime}+1} S_{{k^{\prime}}_j} \bigcap \overline{D}$ with $\{k_j\}_{j=1}^{{l_1}^{\prime}}$ being a subsequence of an increasing sequence $\{{k^{\prime}}_j\}_{j=1}^{{l_1}^{\prime}+1}$.
Then for a suitable small smooth submanifold $D_p \subset {\mathbb{R}}^n$ diffeomorphic to $D^n$ and containing $p$ in its interior ${\rm Int}\ D_p$, 
the preimage ${f_{(D,\{f_j\}_{j=1}^{l_1},\{S_j\}_{j=1}^{l_1},m_{l_1,l_2},m_{l_2})}}^{-1}({\rm Int}\ D_p)$ is a smooth manifold of dimension ${\Sigma}_{j=1}^{l_2} (m_{l_2}(j))+n$ with no boundary and diffeomorphic to the interior of the manifold
$D^{n-{l_1}^{\prime}} \times {\prod}_{j=1}^{{l_1}^{\prime}} D^{m_{l_2} \circ m_{l_1,l_2}(j)+1} \times {\prod}_{j \in {\mathbb{N}}_{l_2}-m_{l_1,l_2}({\mathbb{N}}_{{l_1}^{\prime}})}  S^{m_{l_2} \circ m_{l_1,l_2}(j)}$. 
In the case $p$ is in the open set $D$, for a suitable small smooth submanifold $D_p \subset {\mathbb{R}}^n$ diffeomorphic to $D^n$ and containing $p$ in its interior ${\rm Int}\ D_p$, 
the preimage ${f_{(D,\{f_j\}_{j=1}^{l_1},\{S_j\}_{j=1}^{l_1},m_{l_1,l_2},m_{l_2})}}^{-1}({\rm Int}\ D_p)$ is a smooth manifold of dimension ${\Sigma}_{j=1}^{l_2} (m_{l_2}(j))+n$ with no boundary and diffeomorphic to the interior of $D^{n} \times {\prod}_{j \in {\mathbb{N}}_{l_2}} S^{m_{l_2} \circ m_{l_1,l_2}(j)}$.

We give comments on $D_p${\rm :} $D_p$ is chosen in such a way that the boundary $\partial D_p$ and ${\bigcap}_{j=1}^{{l_1}^{\prime}} S_{k_j}$ intersect in the transversal way.

We have checked that $M$ is non-singular. We can easily check $f(M)=\overline{D}$.

This completes important exposition on the proof.

\end{proof}
Note that the resulting maps of Theorem \ref{thm:1} are regarded as smooth maps locally like so-called {\it moment maps} on complex projective spaces and more general toric symplectic manifolds. For related topics, see \cite{delzant} and see also \cite{buchstaberpanov}.
Note also that around each singular value of the map, it is represented as the product of a Morse-Bott function and the identity map on some manifolds for suitable local coordinates. 

\section{Our main result.}
By applying Theorem \ref{thm:1}, we have the following. We prove this in the next section.

\begin{Thm}
\label{thm:2}
Problem \ref{prob:1} is affirmatively solved in the case $F$ is a closed, connected and orientable surface in $m=3$.
\end{Thm}

This generalizes the cases of $F=S^{m-1}$ ($m \geq 2$) and $F=S^k \times S^{m-k-1}$ ($1 \leq k \leq m-2$ with $m \geq 3$) in the case $m=3$: these two cases are essentially due to Theorem \ref{thm:1}.

The following is closely related to Problem \ref{prob:1} and Theorem \ref{thm:2}. This is also a kind of corollaries easily seen from the proof of Theorem \ref{thm:2}.
\begin{Thm}
\label{thm:3}
Let $F$ be a closed, connected and orientable surface. Let $a<b$ be real numbers.
Then there exist a $3$-dimensional non-singular real algebraic manifold $M$ being compact and connected and a smooth function ${f}_{M,a,b}:M \rightarrow \mathbb{R}$ with the following.
\begin{enumerate}
\item The function ${f}_{M,a,b}:M \rightarrow \mathbb{R}$ is the restriction of the canonical projection ${\pi}_{5,1}$.
\item The image is $[a,b]=\{a \leq t \leq b\}$ where only $a$ and $b$ are singular values of the function.
\item The  preimage ${{f}_{M,a,b}}^{-1}(p)$ is diffeomorphic to $F$ and each $p$. 
\end{enumerate}
\end{Thm}
\section{A proof of our main result and additional comments.}
Hereafter, we need to investigate the {\it Reeb graphs} of smooth functions of a certain class. The notion has already appeared in \cite{reeb}. 
Reeb graphs are graphs and regarded as digraphs naturally.  
We do not explain elementary terminologies, notion and arguments on (di)graphs. We also expect related knowledge.
For a smooth function $c:X \rightarrow \mathbb{R}$ on a manifold with no boundary, we can define the following equivalence relation ${\sim}_{c}$ on $X$: $p_1 \sim p_2$ if and only if they are in a same connected component of $c^{-1}(q)$ for some value $q \in \mathbb{R}$. The quotient space $W_c:=X/{\sim}_c$ is the {\it Reeb space} of $c$ and $W_c$ has the structure of a graph by defining the vertex set as the set of all connected components containing some singular point of $c$ if $X$ is closed with the set of all singular values of $c$ being finite (\cite[Theorem 3.1]{saeki}). This is the {\it Reeb graph} of $c$.
We also have the quotient map $q_f:M \rightarrow W_f$, which is also continuous. 
We can also have the unique function $\overline{f}:W_f \rightarrow \mathbb{R}$ with $f=\bar{f} \circ q_f$. The function $\overline{f}$ is also continuous.

We prove Theorems \ref{thm:2} and \ref{thm:3}. For our proof, main ingredients are explicit construction of explicit moment-like maps. We abuse the notation of Theorem \ref{thm:1}. We can induce the orientation on the Reeb graph canonically and this is regarded as a digraph: the {\it Reeb digraph} of $f$. 
\begin{proof}[A proof of Theorem \ref{thm:2}.]
First, in the case $F$ is a sphere, then it is trivial from our arguments.

We discuss the case where $F$ is a closed, connected and orientable surface of genus at least $1$.

We can have an explicit case for Theorem \ref{thm:1} and our result as follows.
 
We put $n=3$. We put $l_1$ as an integer greater than or equal to $3$. We can choose the corresponding real polynomial functions $f_j$ canonically and suitably from the following hypersurfaces $S_j$ and define our connected bounded region $D$ surrounded by the $l_1$ hypersurfaces $S_j$ satisfying the 2nd condition of Theorem \ref{thm:1}. We can also check the 1st and the 3rd conditions of Theorem \ref{thm:1} easily. We explain $S_j$ precisely and explicitly.

We can put $S_1:=\{(s_1,t,t^{\prime}) \mid t, t^{\prime} \in \mathbb{R}\}$ and $S_2:=\{(s_2,t,t^{\prime}) \mid t, t^{\prime} \in \mathbb{R}\}$ be subspaces (or so-called {\it affine subspaces} of ${\mathbb{R}}^3$) with $s_1<s_2$. We can also choose a cylinder $S_3:=\{(t,t_1,t_2) \mid {t_1}^2+{(t_2-a)}^2={(b-a)}^2, t \in \mathbb{R}\}$ of a circle.
We define the straight line $L:=\{(t,0,a)\mid t \in \mathbb{R}\}$.

We can choose $s_1<s_2$ nicely beforehand to have $s_2-s_1>2(l-3)(b-a)$. In this situation, we can choose an increasing sequence $\{p_j\}_{j=1}^{l-3}$ of real numbers and $S_j$ ($4 \leq j \leq l$) with the following properties.
\begin{itemize}
\item Three relations $p_1-s_1>b-a$, $s_2-p_{l-3}>b-a$ and $p_{j+1}-p_{j}>2(b-a)$ hold. 
\item The cylinder $S_j=\{(t_1,t_2,t) \mid {(t_1-p_{j+3})}^2+{t_2}^2={(b-a)}^2, t \in \mathbb{R}\}$ of a circle is defined. If the cylinder is restricted to the circle $\{(t_1,t_2,a) \mid {(t_1-p_{j+3})}^2+{t_2}^2={(b-a)}^2\}$, then the resulting circle and the straight line $L:=\{(t,0,a)\mid t \in \mathbb{R}\}$ intersect in a one-point set and the tangent spaces of them at the single point agree.
\end{itemize}

FIGURE \ref{fig:1} shows our situation.

\begin{figure}
	\includegraphics[width=80mm,height=80mm]{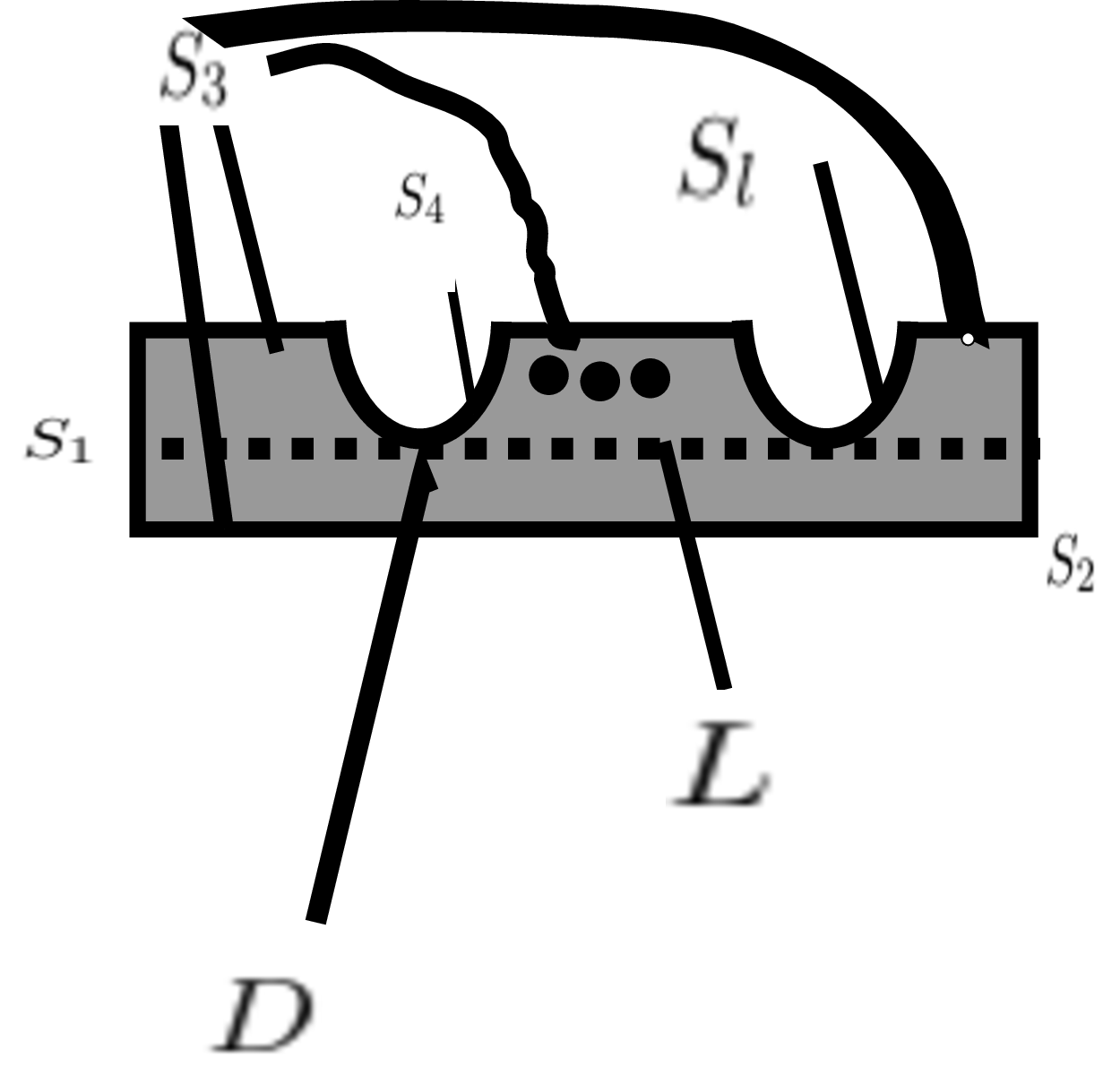}
	\caption{Our desired region $D$, restricted to the subspace $\{(t,t^{\prime},a) \mid t, t^{\prime} \in \mathbb{R}\}$, in Theorem \ref{thm:2}. Note that the space $\{(t,t^{\prime},a) \mid t, t^{\prime} \in \mathbb{R} \}$ is also an affine subspace of ${\mathbb{R}}^3$.}
	\label{fig:1}
\end{figure}

We must not forget defining $m_{l_1,l_2}$ and $m_{l_2}$. We put $l_2:=2$. We can set $m_{l_1,l_2}(i):=1$ for $i \neq 3$, $m_{l_1,l_2}(3):=2$, and $m_{l_2}(i)=0$ for $i=1,2$.

We have our defired map $f_{(D,\{f_j\}_{j=1}^{l_1},\{S_j\}_{j=1}^{l_1},m_{l_1,l_2},m_{l_2})}$. Our desired function ${\tilde{f}}_{\tilde{M},a,b}:\tilde{M} \rightarrow \mathbb{R}$ is defined as the composition of $f_{(D,\{f_j\}_{j=1}^{l_1},\{S_j\}_{j=1}^{l_1},m_{l_1,l_2},m_{l_2})}$ with the projection to the 3rd component of ${\mathbb{R}}^3$ where the type of $F$ is not determined yet. The type of $F$ is determined by considering the projection of  ${{\tilde{f}}_{\tilde{M},a,b}}^{-1}(a) \subset {\mathbb{R}}^5$ to the first component ($x_1 \in \mathbb{R}$), investigating the Reeb graph of the resulting new function  and consulting \cite{gelbukh1} (with \cite{gelbukh2}) and elementary arguments on Morse(-Bott) functions and elementary topological theory of surfaces. The function is regarded as a Morse-Bott function on a closed and connected surface $F$ due to our construction of the map with local properties of the maps. Let the function be denoted by ${f^{\prime}}_{\tilde{M},a,b,F}$. Furthermore, its Reeb digraph is shown to be a graph whose 1st Betti number is $l-2$ and which has exactly two vertices where the function $\overline{f}:=\overline{{{f}^{\prime}}_{\tilde{M},a,b,F}}:W_{{{f}^{\prime}}_{\tilde{M},a,b,F}} \rightarrow \mathbb{R}$ has extrema and exactly $2(l-3)$ vertices where the function does not have extrema. At the latter $2(l-3)$ vertices the function does not have local extrema. The first two vertices are of degree $2$. The $2(l-3)$ vertices are of degree $3$ and to each of these vertices exactly one singular point of the function is mapped.
Thanks to \cite{gelbukh1} (\cite{gelbukh2}) with a kind of fundamental arguments on Morse(-Bott) functions and elementary topological theory of surfaces, 
$F$ is shown to be a closed, connected and orientable surface of genus $l-2$.

Check also FIGURE \ref{fig:2} for our Reeb graph.
\begin{figure}
	\includegraphics[width=80mm,height=80mm]{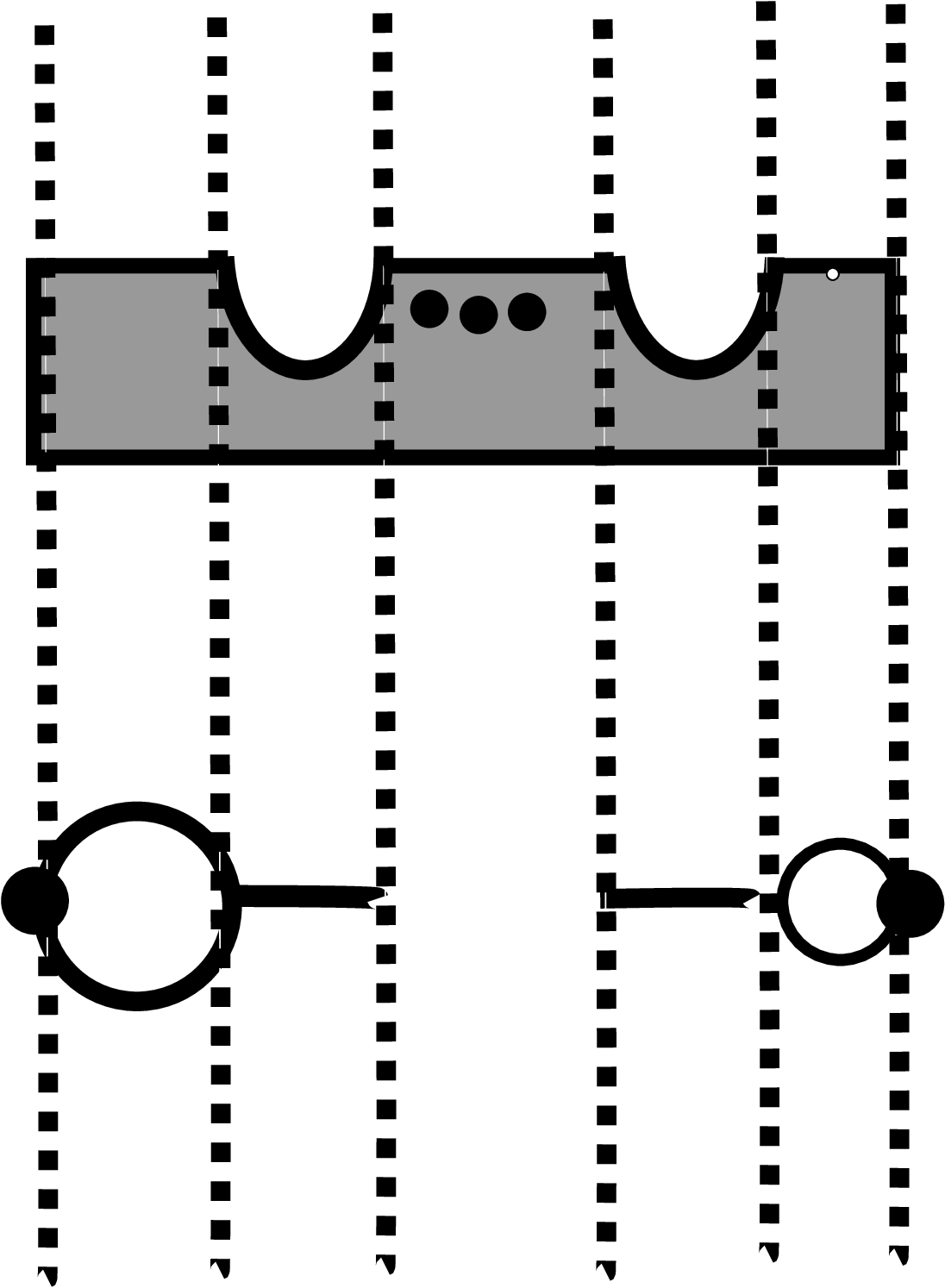}
	\caption{Our Reeb digraph for Theorem \ref{thm:2}.}
	\label{fig:2}
\end{figure}
This completes the proof.
\end{proof}
Theorem \ref{thm:3} can be checked easily from the proof of Theorem \ref{thm:2}. This is regarded as a kind of positive answers to problems on reconstruction of nice smooth or real algebraic functions with prescribed preimages of single points. Such stuides were essentially started by the author in \cite{kitazawa1, kitazawa2}.
\section{Our additional comments and result and future problems.}
\subsection{Our additional comments and result on Problem \ref{prob:1} and Theorem \ref{thm:2}.}
In Theorem \ref{thm:2}, in the case $F$ is of genus $1$, $2$ or $3$. We can give our additional comments and result.

We can have a related explicit case for Theorem \ref{thm:1} and our result as follows.
 
We put $n=3$. We can choose the corresponding real polynomial functions $f_j$ canonically and suitably from the following hypersurfaces $S_j$ and define our connected bounded region $D$ surrounded by the hypersurfaces $S_j$ satisfying the 2nd condition of Theorem \ref{thm:1}. We can also check the 1st and the 3rd conditions of Theorem \ref{thm:1} easily. We explain $S_j$ precisely and explicitly.
We can choose $S_1:=\{(s_1,t,t^{\prime}) \mid t, t^{\prime} \in \mathbb{R}\}$ and $S_2:=\{(s_2,t,t^{\prime}) \mid t, t^{\prime} \in \mathbb{R}\}$ be (affine) subspaces of ${\mathbb{R}}^3$, satisfying $s_1<s_2$, as in the proof of Theorem \ref{thm:2}.
We can choose a cylinder $S_3:=\{(t,t_1,t_2) \mid {t_1}^2+{(t_2-a)}^2={(b-a)}^2, t \in \mathbb{R}\}$ of a circle as in our previous proof. 

We can define the cylinder $S_4=\{(t_1,t_2,t) \mid  t \in \mathbb{R}, (t_1,t_2) \in S_{4,a}\}$ of the uniquely defined
straight line $S_{4,a}:=S_{4,{\rm L},a}:=\{(u\frac{s_1+s_2}{2}+(1-u)s_1,u(b-a),a) \mid u \in \mathbb{R}\} \subset \{(t,t^{\prime},a) \mid t,t^{\prime} \in \mathbb{R}\}$ 
or a circle $S_{4,a}:=S_{4,{\rm C},a} \subset \{(s,s^{\prime},a) \mid s,s^{\prime} \in \mathbb{R}\}$ passing two points $(s_1,0,a)$ and $(\frac{s_1+s_2}{2},b-a,a)$ centered at a point $(p_{4,0,1},p_{4,0,2},a)$ satisfying $p_{4,0,1}<s_1$ and $p_{4,0,2}>b-a$.
We can also define the cylinder $S_5=\{(t_1,t_2,t) \mid  t \in \mathbb{R} ,(t_1,t_2) \in S_{5,a}\}$ of the uniquely defined
straight line $S_{5,a}:=S_{5,{\rm L},a}:=\{(u\frac{s_1+s_2}{2}+(1-u)s_2,u(a-b),a) \mid u \in \mathbb{R}\} \subset \{(t,t^{\prime},a) \mid t,t^{\prime} \in \mathbb{R}\}$ 
or another circle $S_{5,a}:=S_{5,{\rm C},a} \subset \{(t,t^{\prime},a) \mid t ,t^{\prime} \in \mathbb{R}\}$ passing two points $(s_2,0,a)$ and $(\frac{s_1+s_2}{2},a-b,a)$ centered at a point $(p_{5,0,1},p_{5,0,2},a)$ satisfying $p_{5,0,1}>s_2$ and $p_{5,0,2}<a-b$.

We must not forget defining $m_{l_1,l_2}$ and $m_{l_2}$. We put
$l_1=3,4,5$. We define $l_2:=2$ if $l_1=3$ and $l_2=3$ if $l_1=4,5$. We can set $m_{l_1,l_2}(i):=1$ for $i=1,2$, $m_{l_1,l_2}(3):=2$, $m_{l_1,l_2}(4)=m_{l_1,l_2}(5):=3$, and $m_{l_2}(i)=0$ for $i=1,2,3$.
FIGURE \ref{fig:3} shows our situation explicitly.
\begin{figure}
	\includegraphics[width=80mm,height=80mm]{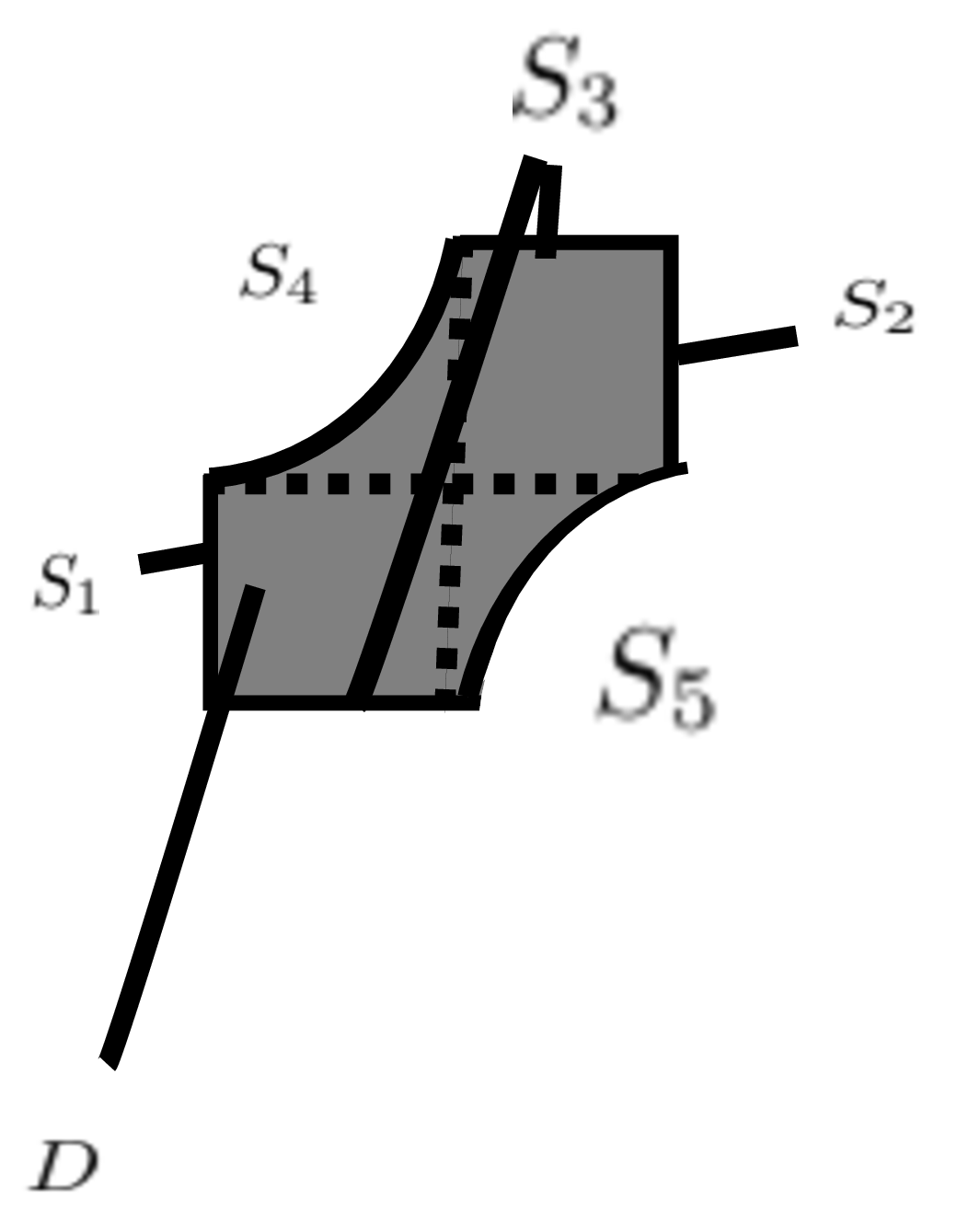}
	\caption{Our desired region $D$ restricted to $\{(t,t^{\prime},a) \mid t, t^{\prime} \in \mathbb{R}\}$ (in the $l_1=5$ case), in the additional study of Theorem \ref{thm:2}. The set $\{(t,t^{\prime},a) \mid t, t^{\prime} \in \mathbb{R}\}$ is also an affine subspace of ${\mathbb{R}}^3$.}
	\label{fig:3}
\end{figure}
We have our defired map $f_{(D,\{f_j\}_{j=1}^{l_1},\{S_j\}_{j=1}^{l_1},m_{l_1,l_2},m_{l_2})}$. Our desired function ${\tilde{f}}_{\tilde{M},a,b}:\tilde{M} \rightarrow \mathbb{R}$ is defined as the composition of $f_{(D,\{f_j\}_{j=1}^{l_1},\{S_j\}_{j=1}^{l_1},m_{l_1,l_2},m_{l_2})}$ with the projection to the 3rd component ($x_3 \in \mathbb{R}$) of ${\mathbb{R}}^4$ where the type of $F$ is not determined yet. The type of $F$ is determined by considering the projection of  ${{\tilde{f}}_{\tilde{M},a,b}}^{-1}(a) \subset {\mathbb{R}}^{m^{\prime}}$ to the first component ($x_1 \in \mathbb{R}$), investigating the Reeb graph of the resulting function and consulting \cite{gelbukh1} (with \cite{gelbukh2}) with elementary arguments on Morse-Bott functions and fundamental topological arguments on surfaces, as in our proof of Theorem \ref{thm:2}: here $m^{\prime}=5$ ($m^{\prime}=6$) in the case $(l_1,l_2)=(2,2)$ (resp. $l_1>2$ with $l_2=3$). 

We investigate these Reeb graphs. We abuse the notation in the proof of Theorem \ref{thm:2} for the functions the Reeb graphs of which we investigate. 
\\
\ \\
Case A. The $l_1=3$ case.  \\
\ \\ In this case, the Reeb graph is homeomorphic to a circle with exactly two vertices. In this case, by the arguments in Theorem \ref{thm:2}, $F$ is of genus $1$. \\
\ \\
Case B. The $l_1=4$ case.  \\
\ \\ In this case, the Reeb graph is shown to be a graph whose 1st Betti number is $2$ and which has exactly three vertices where the function $\overline{f}:=\overline{{{f}^{\prime}}_{\tilde{M},a,b,F}}:W_{{{f}^{\prime}}_{\tilde{M},a,b,F}} \rightarrow \mathbb{R}$ has extrema and exactly two vertices where the function does not have extrema. Here for the latter two vertices the function does not have local extrema. The first three vertices are of degree $2$. The remaining two vertices are of degree $3$ and to each of these vertices exactly one singular point of the function is mapped. In this case, by arguments similar to ones in the proof of Theorem \ref{thm:2}, $F$ is of genus $2$. \\
\ \\
Case. C The $l_1=5$ case.  \\
\ \\ In this case, the Reeb graph is shown to be a graph whose 1st Betti number is $1$ and which has exactly two vertices where the function $\overline{f}:=\overline{{{f}^{\prime}}_{\tilde{M},a,b,F}}:W_{{{f}^{\prime}}_{\tilde{M},a,b,F}} \rightarrow \mathbb{R}$ has extrema and exactly two vertices where the function does not have extrema: for the latter vertices the function does not have local extrema. The first two vertices are of degree $2$. The remaining two vertices are of degree $2$ and to each of these vertices exactly two singular points of the function are mapped. In this case, by arguments similar to ones in the proof of Theorem \ref{thm:2}, $F$ is of genus $3$. \\

Check also FIGURE \ref{fig:4} for such a Reeb graph.
\begin{figure}
	\includegraphics[width=80mm,height=80mm]{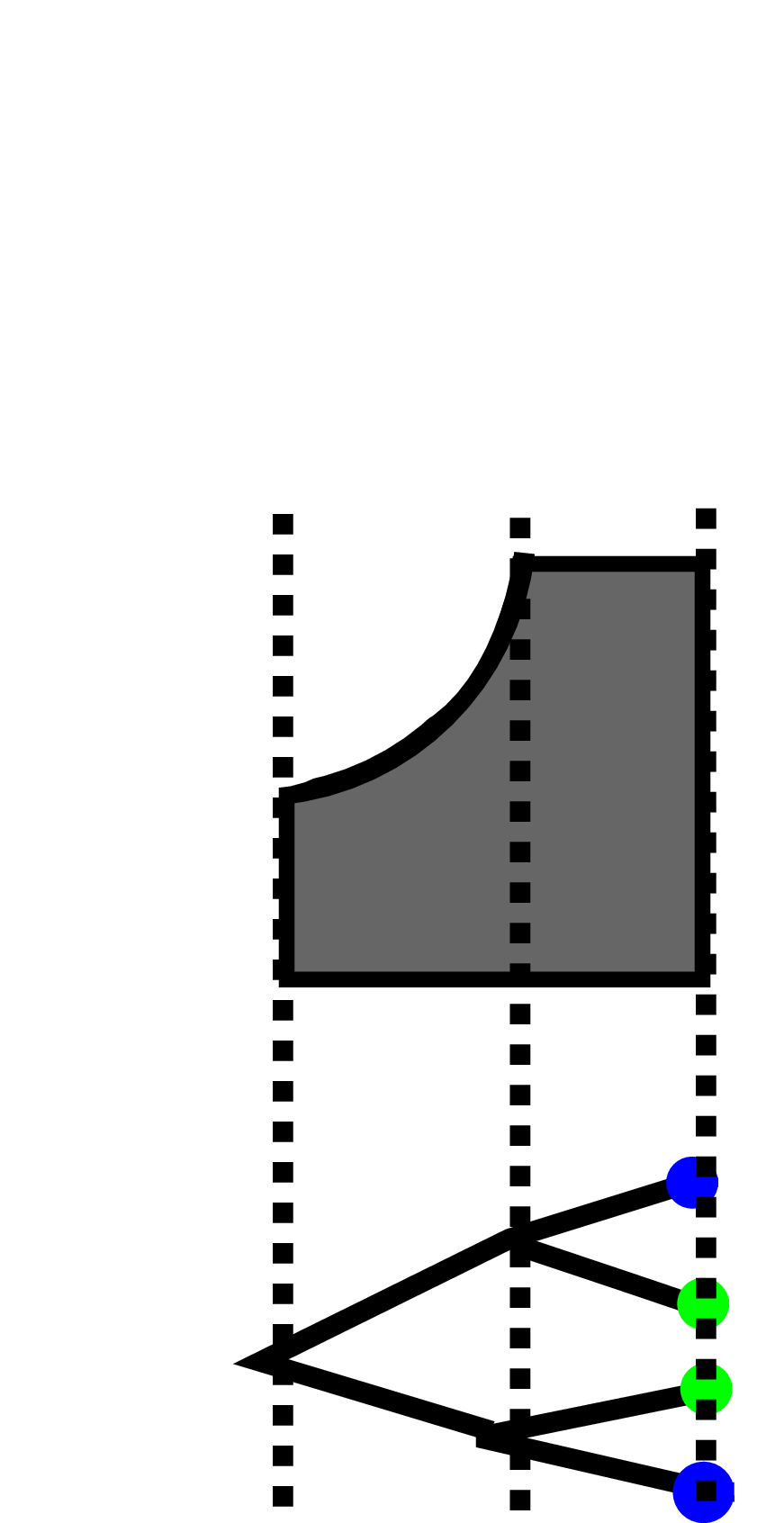}
	\caption{Our Reeb digraph in the case $l_1=3$: the two blue dots are identified and the two green dots are also identified.}
	\label{fig:4}
\end{figure}
\subsection{Our future problems.}

\begin{Prob}
\label{prob:2}
Do our maps of Theorem \ref{thm:2} topologically equivalent to important maps in \cite{kitazawa1, kitazawa5}?
\end{Prob}
These maps seem to be very similar topologically where we do not reach related positive idea yet.
\begin{Prob}
\label{prob:3}
Can we find generalizations of moment-like maps with respect to the preimages of single points. In short, can we generalize Theorem \ref{thm:1} with respect to preimages of single points. For example, in the case the difference of the manifold of the domain and that of the target is $2$, the preimages are either diffeomorphic to $S^2$ or $S^1 \times S^1$. It may be important and difficult to construct such maps locally or around singular values (of the maps). 
This may be regarded as a higher dimensional version of Problem \ref{prob:1}: the case where the dimension of the space of the target is higher than $1$. 
\end{Prob}
Related to this, for example, in the previous subsection, we consider the case $l_1=3$ again in a revised form. Roughly, we do as follows.

We put $n=4$. We replace $S_1$ and $S_2$ into the cylinder of a circle as has been done in several scenes here and change the notation for this new subset to $S_1$, and change the notation for "$S_3$" to "$S_2$", without changing the subset.  We put $l_2:=2$. We can set $m_{l_1,l_2}(i):=i$ for $i=1,2$ and $m_{l_2}(i)=0$ for $i=1,2$.
We can choose the cylinders $S_1:=\{(t_1,t,t_2,t^{\prime}) \mid {(t_1-\frac{s_1+s_2}{2})}^2+{t_2}^2=\frac{{(s_2-s_1)}^2}{4}, t \in \mathbb{R},  t^{\prime} \in \mathbb{R}\}$ 
 and $S_2:=\{(t,t_1,t^{\prime},t_2) \mid {t_1}^2+{(t_2-a)}^2={(b-a)}^2, t \in \mathbb{R},  t^{\prime} \in \mathbb{R}\}$. 

We construct the moment-like map similarly. We consider the composition of the resulting moment-like map into ${\mathbb{R}}^4$ with the projection to the (ordered) pair $(x_3,x_4) \in {\mathbb{R}}^2$ of the 3rd and the 4th components of ${\mathbb{R}}^4$. This is locally regarded as a moment-like map into ${\mathbb{R}}^2$. It may be natural to expect the case where the preimages $F$ of single points in the interior of the image are of a higher degree. We do not have positive idea yet.

For smooth maps locally moment-like maps, see also \cite{kobayashi, kobayashiyamamoto}. 
There local forms of moment maps on complex projective spaces and more generally toric symplectic manifolds are regarded as main objects. They are investigated as smooth maps in the studies.
These studies are concerned with visualizations of geometric objects and important objects in singularity theory and understanding explicit smooth maps and manifolds of the domains through such visualizations.
Note again that explicit systematic construction of local and global real algebraic maps were essentially started by the author \cite{kitazawa2, kitazawa6, kitazawa7, kitazawa8}: terminologies such as moment-like maps are due to the author \cite{kitazawa7}. 

\begin{Prob}
\label{prob:4}
After we succeed in construction of maps in Problem \ref{prob:3}, can we investigate linear perturbations of these maps and have so-called generic maps in singularity theory. 
\end{Prob}
Problem \ref{prob:4} asks whether we can generalize some of main ingredients of \cite{kobayashi, kobayashiyamamoto}.\\
\ \\
\noindent {\bf Conflict of interest.} \\
The author is a researcher at Osaka Central
Advanced Mathematical Institute (OCAMI researcher), supported by MEXT Promotion of Distinctive Joint Research Center Program JPMXP0723833165. He is not employed there. Our study thanks this. \\
\ \\
{\bf Data availability.} \\
No other data are not generated. 

\end{document}